\documentclass[a4paper, oneside, 12pt]{amsart}
\usepackage[a4paper, total={6in, 8in}]{geometry}
\usepackage[colorlinks=true, citecolor=red]{hyperref}
\usepackage{svg}
\usepackage{float}

\newcommand{\citep}{\cite}
\usepackage[numbers,sort&compress]{natbib}

\usepackage{color}
\newcommand{\mb}{\mathbb}
\newcommand{\mc}{\mathcal}

\newcommand{\C}{\mathbb{C}}
\newcommand{\Z}{\mathbb{Z}}
\newcommand{\CC}{\mathbb{C}}
\newcommand{\NN}{\mathbb{N}}
\newcommand{\DD}{\mathbb{D}}
\newcommand{\ZZ}{\mathbb{Z}}
\newcommand{\HH}{\mathbb{H}}
\newcommand{\HHH}{\mathcal{H}}
\newcommand{\RR}{\mathbb{R}}
\newcommand{\fr}{\frac}
\newcommand{\eps}{\varepsilon}
\newcommand{\m}{\wedge}
\newcommand{\til}{\tilde}

\newtheorem{theorem}{Theorem}[section]

\newtheorem{defn}{Definition}[section]
\newtheorem{proposition}{Proposition}[section]
\newlength\tindent
\title[Collection of results on exit time of Brownian motion]{A collection of results relating the geometry of plane domains and the exit time of planar Brownian motion, II} 
\author{Greg Markowsky}
\author{Clayton McDonald}
\thanks{Email: greg.markowsky@monash.edu, clayton.mcdonald@monash.edu (corresponding author)}

\begin{document}
 \begin{abstract}
This paper is the sequel to another with the same name (Buttigieg et al., Comput. Methods Funct. Theory, 2023), and is concerned with results of the same type. We deduce a result on the moments of the exit time of Brownian motion from domains whose boundary curve is replaced by a dashed line, and from domains arising from a periodic tiling of the plane. We also give a construction of a type of domain which is similar to a wedge domain, but the behaviour of whose exit time moments answer several questions that had been speculated upon.
 \end{abstract}
\maketitle

\def\changemargin#1#2{\list{}{\rightmargin#2\leftmargin#1}\item[]}
\let\endchangemargin=\endlist

\section{Introduction} \label{intro}

Let $D$ be a domain in $\RR^2$, and let $T=T(D)$ be the first time that a Brownian motion
leaves $D$; that is, $T(D) = \inf\{t\geq 0: B_t \notin D\}$. The distribution of
this stopping time carries a great deal of information about the size and shape
of the domain, as well as the location of the initial point $a$. This idea
seems to have been developed in the seminal papers \cite{burk, davis}, where a
discussion of the conformal invariance of Brownian motion led naturally to the
suggestion that the distribution of $T(D)$ has fundamental connections with the
theory of conformal maps. This connection has been explored by many authors in the last few decades. A recent paper with the same title as this one, \cite{boudbuttme}, contained a collection of loosely related results on this connection. This paper is also of the same type, and our main focus in this one is on domains which are not simply connected. There does not seem to be a reference which surveys exit times and their connections to complex analysis, so we begin with a review of the literature already present. 

Conformal maps are difficult to work with when the
domains in question are not simply connected, so it is natural that most of the
research on this topic (such as \cite{banuelos1993conditioned,
banuelos2024conjecture, banuelos2001first, banuelosdrum, banuelos2005sharp,
dimitri, mecomb, usexitmax, mejmaa, carroll2001old, carroll2006brownian, mac,
meecp, coffee, karafyllia2021property}) has dealt primarily with
domains that were simply connected. 

The expectation of $T(D)$ is easier to work with than its distribution
function, and is obtainable in several different way. Determining $E[T(D)]$ is equivalent to
solving the torsion problem in PDE theory (see \citep{banuelos2002torsional}), and
it can also be determined from the coefficients of a conformal map from the
unit disk onto $U$ (see \citep{banuelosdrum}). Further research in this area
(such as \citep{banuelos1993conditioned,banuelos2024conjecture,carroll2001old,
coffee, meecp}) has generally exploited the equivalence of these various
quantities.

By contrast, the moments $E[T(D)^p]$ are not generally computable for $p \neq
1$, however Burkholder(\citep{burk}) established methods for estimating them,
and also established an important relationship between them and the Hardy norm
of analytic functions. These observations have led also to a number of works by
other authors, such as
\citep{banuelos2005sharp,banuelos2020bounds,carroll2006brownian,kim2021quantitative,mecomb,mejmaa,usexitmax}.

The tail of the distribution of $T(D)$, i.e. the function $t \to P(T(D)>t)$, is also
open to analysis, as it is the solution to the heat equation with certain
boundary values on the domain. However, only in very rare cases does this PDE
have a known, explicit solution. As a result, research on this area has mainly focused on asymptotics; see \cite{deblassie1987exit, banuelos2001first, betsakos2022duration, dimitri, de1988remark, banuelos1997brownian}.

In this paper, we are interested in developing results on certain domains that
are not simply connected, and for which analytic methods are often difficult to
apply. There have been a few papers that have touched upon the non-simply
connected case. In the papers \citep{banuelos2002torsional, banuelos2020bounds,
betsakos2022duration, kim2021quantitative} results in
$\RR^2$ are generally special cases of more general results in $\RR^n$, and often
employ general analytic techniques to obtain bounds. There have been relatively few papers specific to the two dimensional case focusing on domains which are not simply connected and which make use of complex analytic techniques, although an exception to this may be developing in the recent developments on the Schottky-Klein prime function (see \cite{crowdy2020solving,crowdy2016schottky}). In relation to Brownian motion, this tool has thus far been mainly used to calculate the distribution of harmonic measure (see \cite{mahenthiram2024harmonic, mahenthiram2025computing, green2022harmonic,snipes2016harmonic, barton2014new, snipes2005realizing, snipes2008convergence, green2024towards}), but harmonic measure is a close relation to exit time, so it may be anticipated that the prime function will develop into a valuable tool for studying exit times in the future.

The prime function has thus far been applied only to domains with finite, and in many cases relatively small, multiplicity, since these are domains for which it provides an explicit formula. This brings us naturally to the main subject of this paper. We are only interested in the two-dimensional case, and are primarily interested in domains of a certain type with infinite multiplicity, which means that the prime function will not be of use to us. In contrast to most of the papers above, which have been heavily analytic, the methods in this
paper are probabilistic in nature, as the main tool in our proofs is a stopping time
argument. Since our results can be translated into purely analytic ones (about Hardy numbers), it would be interesting to see how analytic arguments might be used to prove them.

One of our main results is concerned with the following question. Suppose we have a domain where the moments of Brownian exit time are well understood, and whose boundary is a curve. Suppose now that we put gaps in the boundary curve - essentially replacing the boundary curve with a dashed line - which allow Brownian motion to escape to the outside of the domain. How does this change the moments of the exit time? A natural, concrete example is the following. If we denote the upper half-plane by $\HH$, then it is well known that $E[T(\HH)^p] < \infty$ precisely when $p < 1/2$. How does this change when the boundary, the real line, is replaced by a dashed line? We will be able to solve this question, and give a somewhat general result, in the next section.

Another type of domain we will consider is one that is periodic, in that it can be realized as a
tiling of the plane by identical boundary components. We will see again that
stopping time arguments allow us to obtain strong bounds on the moments of exit
times. This is discussed in Section \ref{periodic}.

Finally, in Section \ref{achievecrit}, we give an example which answers two questions which have been posed to the first author by a number of colleagues, which we now describe. To describe the first question, we will define the {\it Brownian-Hardy number} of the domain $D$ (the reason for this name will become clear later) by

\begin{align*}
    bh(D) := \sup \{p>0: E[T(D)^p] <\infty\} = \inf \{p>0: E[T(D)^p] =\infty\}.
\end{align*}

Most domains $D$ for which $bh(D)$ is easily calculated and for which $0<bh(D)<\infty$ satisfy $E[T(D)^{bh(D)}]=\infty$, and the natural question is whether this must always be the case. We give an example that shows it need not be, that it is indeed possible that $E[T(D)^{bh(D)}]<\infty$.

The second question is concerned only with the first moment, $E_a[T(D)]$. If we consider this to be a function of the initial point $a$, in other words $h(a) = E_a[T(D)]$, then it is known by standard methods (such as Dynkin's formula) that, ignoring technical issues, $h$ should be a positive solution to

\begin{enumerate}
    \item[$(i)$] $\Delta h = -2$

    \item[$(ii)$] $h(a) = 0$ for $a \in \delta D$
\end{enumerate}

These conditions are enough to uniquely determine $h$ in many cases, for instance when $D$ is bounded and simply connected, but it is not enough when $D$ is unbounded. To give an easy example which illustrates this we may let $D = \{-\frac{\pi}{2} < Re(z) < \frac{\pi}{2}\}$, and then note that
$h(x,y) = \frac{\pi^2}{4} - x^2 + Ce^x \cos(y)$ satisfies $(i)$ and $(ii)$ for any constant $C \geq 0$ (note that $e^x \cos(y)$ is harmonic and equal to $0$ on $\delta D$). Of course, the correct choice here must be $C=0$, since it is clear that $E_{x+yi}[T(D)]$ cannot depend on $y$, but it should be equally clear that a more general method of choosing the correct solution is required.

Happily, Burkholder provided a completely satisfying answer, in \cite{burk}, and an analytic proof of the corresponding analytic statement was given in \cite{markowsky2013method}, where the method was used to calculate the expected exit times of a number of domains. Burkholder's insight is that one must look for a function with at most quadratic growth. In particular, he proved the following.

\begin{theorem} \label{burk_thm}
    Suppose $h$ is a positive function on a domain $D$ satisfying $(i)$ and $(ii)$ above. Suppose also that

    \begin{enumerate}
    \item[$(iii)$] There is a constant $C>0$ such that $h(z) \leq C(1+|z|^2)$ for all $z \in D$.
\end{enumerate}

Then $h(a) = E_a[T(D)]$.
\end{theorem}

In other words, if we find a function $h(a)$ which satisfies the required pde, and has quadratic growth, then it must be equal to $E_a[T(D)]$. However, this leads to a natural question: does the function $a \to E_a[T(D)]$ always have quadratic growth? We will show that it does not, as the domain we construct to answer the first question above will satisfy

$$
\sup_{a \in D}\frac{E_a[T(D)]}{1+|a|^2} = \infty.
$$

\section{Preliminaries}
The $p$-th Hardy space $H^p$ is the class of holomorphic functions defined on the unit disk $\DD$ which satisfy  
\begin{align*}
    \Vert f\Vert_p^p := \frac{1}{2\pi}\sup_{r\leq 1} \int_0^{2\pi} |f(re^{it})|^p dt < \infty
\end{align*}
The norm $\Vert . \Vert_p$ is the $p$-th Hardy norm. Fixing $f$, the function $p\mapsto\Vert f\Vert_p$ is 
increasing, and consequently all the Hardy spaces are nested. As previously mentioned, the study of members in
these spaces has a relationship with the expected exit time of a domain. We
will briefly outline this relationship.

The \textit{Hardy number of a holomorphic function $f$} on the disk, denoted $h(f)$,
is defined by the following:

\begin{align*}
    h(f) := \sup_{p>0} \{p: \Vert f\Vert_p <\infty\} = \inf_{p>0} \{p: \Vert f\Vert_p =\infty\}
\end{align*}

Note that $h(f)$ can be characterized, when finite, as the number such that $f\in H^{h(f)-\eps}$ and $f\notin H^{h(f)+\eps}$ for 
all $\eps>0$. The question of what happens at the critical value $h(f)$ is an interesting one, and we will address this later in the paper. 

The {\it Hardy number of a domain} $D$ is defined via
    \begin{align*}
        h(D) := \inf_{f\in H(\DD, D)} h(f)
    \end{align*}
    where $H(\DD, D)$ are all holomorphic maps from $\DD$ into $D$. Note that this infimum is always attained by the universal covering map of $D$, and the Hardy number can be alternatively defined in that way. Note that this includes the case when $D$ is simply connected and $f$ is conformal, hence the connection with conformal maps.\\

Recalling that our primary interest in this paper is in moments of exit times, we will now discuss the connected probabilistic theory. We will use the notation $B_t = R_t + i I_t$ to denote planar Brownian motion, so $R_t, I_t$ are independent, one-dimensional Brownian motions. We further will use $a$ always to denote the initial point of $B$, i.e. $B_0 = a$ almost surely. Among other remarkable facts, it was shown in \cite{burk} that the finiteness of
$E_{a}[\tau(R)^{p}]$ and $E_{a}[|B_{\tau(R)}|^{2p}]$ does not
depend on $a$. We are therefore free to make statements such as "$E[\tau(R)^{p}]< \infty$" and "$E[\tau(R)^{p}]=\infty$" without specifying $a$. It was also shown by Burkholder that 
\begin{align*}
    E[T(D)^{p}] < \infty \Longleftrightarrow \Vert f\Vert_{2p} < \infty,
\end{align*}

where again $f$ is the universal cover from $\DD$ onto $D$ (the domains which don't have a universal cover from the disk, namely $\CC$ and the once-punctured plane, trivially have $E[T(D)^{p}] = \infty$ for all $p$, since in fact $T(D) = \infty$ almost surely in both cases).


We see that the Hardy number of a domain can actually be defined in terms of the moments of exit time of Brownian motion, and this has several advantages. To begin with, the exit time is a highly intuitive quantity, and allows one to sidestep the notion of a universal cover, which is an abstract and advanced topic. Furthermore, it allows one to extend to definition so that it applies more generally, not only to domains. To make this precise, let us recall the details of what is commonly referred to as the "conformal invariance of Brownian motion", although it might be more accurately be called "analytic invariance" or "holomorphic invariance", as injectivity does not play a role. Given a nonconstant analytic function $f$ on a domain $D$, let

\begin{align} \label{skyfall}
    \sigma_t :=  \int_0^{T(D) \wedge t} |f'(B_s)|^2 ds. 
\end{align}

It is not hard to see that $\sigma_t$ is strictly increasing and continuous a.s., and the same is therefore true of $C_t := \sigma_t^{-1}$. Levy's Theorem (see \cite{bass, davis, durBM}) now states that the process $\hat B_t = f(B_{C_t})$ is a planar Brownian motion, however some attention must be paid to the time interval upon which it is defined. $\sigma$ maps the random interval $[0,T(D)]$ onto another random interval $[0,\sigma_{T(D)}]$, and it is upon precisely this random interval that $X_t$ is defined. Furthermore, $\sigma_{T(D)}$ is a new stopping time, but it is not necessarily the exit time of a domain. To give an example of this, suppose $D$ is the infinite strip $\{-2\pi < Im(z) < 2\pi\}$ and $f(z) = e^z$. The stopping time $\sigma_{T(D)}$ can now be characterized as the first time that the argument of $\hat B_t$ is $\pm 2\pi$, or equivalently the first time that $\hat B_t$ hits the positive real axis, having wound once in either direction about the origin. It is not the exit time of a domain.

Stopping times of the form $\sigma_{T(\DD)}$ are also covered by Burkholder's results. To state this precisely, we will let 

$$
\HHH = \{\tau: \tau = \sigma_{T(\DD)} \text{, with $\sigma$ defined by (\ref{skyfall}) for some $f$ analytic on $\DD$}\}
$$


Burkholder's results now establish that $E[\sigma_{T(D)}^{p}] < \infty \Longleftrightarrow \Vert f\Vert_{2p} < \infty$, where $\sigma$ is the time change of the Brownian motion associated to $f$.
This justifies the name of the Brownian-Hardy number given earlier, which for any stopping time $\tau \in \HHH$ is defined to be 

\begin{align*}
        bh(\tau) := \sup_{p>0}\{p: E_z [\tau^p] < \infty\} = \inf_{p>0}\{p: E_z [\tau^p] = \infty\}.
\end{align*}

\if2 
The $p$-Hardy norm can be seen as the expectation of the $p^{th}$
moment of $f(B_{T(\mathbb{D})})$ as $B_{T(\mathbb{D})}$ is uniformly
distributed on the circle. Thus, the equivalence of the finiteness
of the $p^{th}$ moment of the stopped Brownian motion and the $2p^{th}$
Hardy norm of a conformal map can be seen as a consequence of the
conformal invariance of Brownian motion. 
\fi

In many cases we will want $\tau$ to be $T(D)$, the exit time of a domain $D$. In that case we will use the shorthand $bh(D) := bh(T(D))$. Note that Burkholder's result states that $2bh(D) = h(D)$ for any domain $D$, so that this new definition agrees with the classical Hardy number (up to a multiplicative constant). It is clear that the Brownian-Hardy number is monotonic, in the sense that if $\tau_1 \leq \tau_2$ a.s. then $bh(\tau_1) \geq bh(\tau_2)$, and this monotonicity caries over to domains: if $U \subseteq D$ then $bh(U) \geq bh(D)$. This implies that if $D$ is bounded then $bh(D) = \infty$, since $D$ is contained in some disk, which is then the image of the unit disk under a linear map. The upper-half plane has $bh(\HH) = 1/2$, which follows as 
$(1-z)/(1+z)$ is in $H^{1-\epsilon}$. If $D=\{|z|>1\}$, the complement of the closed unit disk, then $bh(D) = 0$; this can be seen by verifying that the covering map $e^{(1-z)/(1+z)}$ is not in any Hardy space. 

These examples are all exit times of domains, but it is worth understanding a stopping time that is not of that form. Let $\alpha>0$ be given, and suppose that $B_0 = 1$ a.s. Let $\tau_\alpha = \inf\{t \geq 0: arg(B_t)= \pm \alpha\}$; note that $arg(B_t)$ is a well defined, continuous process with $arg(B_t) = 0$ a.s. When $\alpha \leq \pi$, $\tau_\alpha$ is the exit time of a wedge, but when $\alpha > \pi$ it is no longer of that form. The exit time of a wedge was treated by Burkholder in \cite{burk}, and he showed there that $bh(\tau_\alpha) = \frac{\pi}{4\alpha}$; in fact, he stated this only for $\alpha \leq \pi$, the range that corresponds to the exit time of a domain, but his result holds equally well for all $\alpha$. This follows by noting that $\tau_\alpha$ is the image of the exit time of the disk under the analytic function $f(z) = \Big(\frac{1-z}{1+z}\Big)^{2\alpha/\pi}$, and this function is in $H^{2p}$ precisely when $4\alpha p/\pi < 1$.

Essentially we are defining a classical quantity from analysis, in this case the Hardy number of a domain, in terms of the exit time from that domain of Brownian motion. This general idea has been used before, in particular in \cite{markowsky2017planar} for the Green's function, and in \cite{Markowsky_2018} for harmonic measure. Defining these quantities in this way extends the classical definitions so that they cover more cases, and also is of value in proving results on the classical objects themselves; for example, it was used to prove a version of the Riemann mapping theorem in \cite{markowsky2017planar}, and the germ of the idea was also used in \cite{markowsky2018remark} to prove the regularity of the Dirichlet problem for simply connected domains in the plane. It will be interesting to see whether other examples of this method can be developed in the future.

\if2 To see this, we show for all $p>0$ the
function $|z|^p$ has no harmonic majorant in $D$. From a theorem in \cite{ransford} this is equivalent to showing
\begin{align*}
    \int_D g_D(z, w) |w|^{p-2} dw = \infty
\end{align*}
For all $z\in D$, where $g_D$ is the Green's function of $D$. The above is simple to show. \fi 

\section{Domains with boundaries replaced by dashed curves}

For simplicity in the following proofs, we will denote $\Vert T \Vert_p = E_a[|T|^p]^{1/p}$. We are only concerned with finiteness
of these quantities, so the value of $a$ is irrelevant, and we suppress it in the notation.

\begin{theorem} \label{halfplane}
    Let $D$ be the half plane, but with the real line replaced with dashed line segments
    of length $r$, with distance $x$ between the centers of these segments (here $x > r$. That is,
    \begin{align*} D := \mb{C} \setminus \left(\mb{R}\cap
        \bigcup_{n\in\mb{Z}} \{|z-nx|\leq r\}\right).\end{align*} Then we have
        $bh(T(D)) = bh(T(\HH)) = 1/2$, independent of $x$ and $r$.
\end{theorem}
\begin{proof}
   We first define $K = \{z:\Im{z} = 1\}$ and $\tilde K = \{z:\Im{z} = -1\}$. We will assume $B_t$ starts at $z\in K$. 
   Define the stopping rules $\tau_0 = 0$, $\tau_{2j+1} =\inf\{t>\tau_{2j}:B_t\in K\}$, and $\tau_{2j} = \{t>\tau_{2j-1}: B_t\in \tilde K\}$. See Figure \ref{half-plane-fig} for the behavior of these stopping times for a fixed Brownian path. Let $\tilde \tau_j = 
   \tau_j\wedge T(D)$. We have $\lim_{j\to\infty} \tilde\tau_j = T(D)$. By
   Minkowski's inequality,
\begin{equation} \label{viking}
        \Vert\tilde \tau_k\Vert_p^p \leq \sum_{j=1}^k \Vert\til\tau_j - 
        \til\tau_{j-1}\Vert_p^p.
\end{equation}   

   The random variable $\tau_j - \tau_{j-1}$ is
   equivalent to the first exit time from either the half-planes of  $\{z:\Im{z} > -1\}$ or 
   $\{z:\Im{z}<1\}$. Therefore $E(\tau_j-\tau_{j-1})^p <\infty$ for all $p < 1/2$. By the law of total probability,
   $$
        \Vert\til\tau_j - 
        \til\tau_{j-1}\Vert_p^p \leq  E_a[(\tau_j-\tau_{j-1})^p|\tau_{j-1}\leq T(D)]P_a(\tau_{j-1}\leq T(D))
   $$
   
   The event
   $\{\tau_{j-1}\leq T(D)\}$ encodes the probability Brownian
   motion passes from $K$ to $\tilde K$ $j-1$ times without hitting the 
   boundary of $D$. For $a \in \{z:\Im{z} \leq 1\}$, say, the probability $P_a(\tau_{1}\leq T(D))$, which is the probability of hitting $\tilde K$ before exiting $D$ when starting at $a$, can be bounded above by $u(a) := P_a(B_{T(\HH)} \in \bigcup_{n\in\mb{Z}} \{|z-nx| \leq r\})$, which is the probability that the Brownian motion exits the upper half plane on one of our dashed line segments. $u(a)$ is well-known to be harmonic (and therefore continuous) on $\HH$, and it is also periodic of period $x$ because the same is true of  $\bigcup_{n\in\mb{Z}} \{|z-nx| \leq r\}$. It is therefore strictly bounded away from 1, so let us say that $P_a(\tau_{1}\leq T(D))<\alpha<1$. Then it follows that $P_a(\tau_{j-1}\leq T(D)) < \alpha^{j-1}$. Furthermore, the invariance of the entire picture under horizontal translations shows that $E_a[(\tau_j-\tau_{j-1})^p|\tau_{j-1}\leq T(D)]$ is equal to a deterministic constant independent of $a$, call it $C_p$. Thus,
   $$
   \Vert\til\tau_j - 
        \til\tau_{j-1}\Vert_p^p \leq C_p \alpha^{j-1}.
   $$
   Passing to the limit in $\eqref{viking}$, we have
   $$
   E[T(D)^p] \leq C_p \sum_{j=1}^\infty \alpha^{j-1} < \infty,
   $$
   if $p<1/2$. That is, $bh(D) \geq 1/2$. On the other hand, $\HH \subseteq D$, and $bh(\HH) = 1/2$, so $bh(D) \leq 1/2$. The result follows.
\end{proof}
\begin{figure}
 \includesvg[width=10cm, height=8cm]{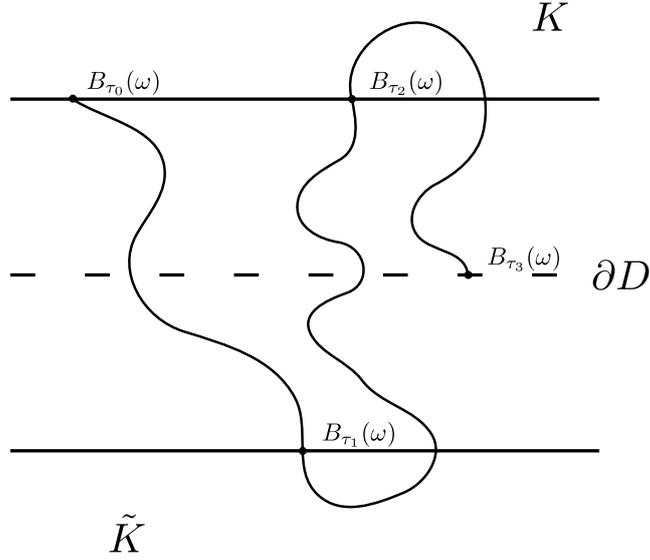}
 \caption{Definition of the Stopping Times $\{\tau_j\}$}\label{half-plane-fig}
\end{figure}
In fact, this result admits a significant generalization if we extract only the parts of the proof that are needed. Before we state the generalization, however, we must make a few more definitions.

\begin{defn}
(Jordan curve)
A Jordan curve is a closed curve in $\CC$ which is the homeomorphic image of a circle.
\end{defn}
\begin{defn}
(Jordan domain)
A domain whose boundary is a Jordan curve is called a Jordan domain.
\end{defn}
Note that Jordan domains are automatically simply connected. The question of
calculating the Hardy number of a Jordan domain is not interesting, because a
Jordan domain is bounded and it follows easily that its Hardy number is
infinite. We therefore adjust the definitions as follows.

\begin{defn}
(Jordan{*} Domain)
A Jordan{*} curve is a curve $\hat C$ in the Riemann sphere $\hat \CC$ which is the image of Jordan curve $C$ under a Mobius transformation which takes a point of $C$ to $\infty$. A Jordan* domain is a (necessarily unbounded) domain in $\CC$ whose boundary is a Jordan* curve.
\end{defn}

Essentially we are modifying the class of Jordan domains to require $\infty$ to
be a boundary point, and the Hardy number of Jordan* domains is no longer
trivial. Alternatively, Jordan* domains could be defined to be domains whose
boundary contains $\infty$ and is homeomorphic (as a set in the Riemann sphere)
to a circle. Note that the Jordan Curve Theorem shows that Jordan* domains naturally come in pairs, since the complement of a Jordan curve is a pair of domains; we will refer to such a pair as {\it complementary Jordan* domains}. 

Jordan* domains were introduced in \cite{mejmaa}, where they were proposed as a natural class of domains upon which to study the moments of Brownian exit times, and were also considered in
\cite{boudbuttme}, where it was shown that it is possible for complementary Jordan* domains to be both simultaneously large or both simultaneously small at infinity when viewed according to Brownian exit time moments. In this paper, our interest is in the following question: suppose that $D,\tilde D$ are complementary Jordan* domains, whose shared boundary is a Jordan* curve $\hat C$; if we form a new domain $U$ by putting gaps in $\hat C$, so that Brownian motion is able to pass through $\hat C$, and taking $U$ to be the union of $D$ and $\tilde D$ together with these gaps, then what can we say about $m(U)$? It is clear that $m(U) \leq \min\{m(D), m(\tilde D)\}$, because $D, \tilde D \subseteq U$, but when can equality occur? An answer to this question is given by the following theorem.

\begin{theorem}
    Suppose $D$ is a Jordan* domain with boundary given by a Jordan* curve $\hat C$, and with complementary Jordan* domain $\tilde D$. Define a new domain $U$ by $U:=D \cup \tilde D \cup V$, where $V$ is an open subset of $C$. Suppose there are closed sets $K \subset D, \tilde K \subset \tilde D$ with the following properties.

    \begin{itemize}
        \item $\sup\{E_a[T(K^c)^p] : a \in \tilde K\} < \infty$.
        \item $\sup\{E_a[T(\tilde K^c)^p] : a \in K\} < \infty$.
        \item $\inf\{P_a(T(U)<T(\tilde K^c)): a \in K\} > 0$ or $\inf\{P_a(T(U)<T(K^c)): a \in \tilde K\} > 0$.
    \end{itemize}

    Then $bh(U) = \min\{bh(D), bh(\tilde D)\}$.
\end{theorem}

{\bf Proof:} (sketch) The proof is analogous to that of Theorem \ref{halfplane}, with $\hat C$ playing the role of $\RR$ there, and $K, \tilde K$ playing the roles of the lines $\{\Im{z} = \pm 1\}$. We can start the Brownian motion at a point in $K$, then let the stopping times $\tau_n$ be defined as before, the successive hitting times of $K$ and $\tilde K$. Let $\tilde{\tau}_j := \tau_j \m T(D)$. Split $\tau_n$ as in \eqref{viking}, and bound in the same way as before. This results in a convergent geometric sum, and the result follows. \qed \\

A basic example that satisfies the conditions above is that of the wedge. Let $\alpha \leq \pi/2$ and define $C = \{z: z = re^{\pm i\alpha}, \; r\geq 0\}$. We will denote the two complementary wedges by $W_1, W_2$. Suppose $W_1$ is the smaller of the two, meaning
it has an opening angle $2\alpha$. It follows $W_2$ has an opening angle $\pi - \alpha \geq \alpha$. From Burkholder's paper \cite{burk} we have $bh(W_1) = \frac{\pi}{4\alpha}$, and $bh(W_2) = \frac{\pi}{4(\alpha-\pi)}$. In this case $bh(W_1) \geq bh(W_2)$.
Now define $K = W_1^* + 1 := \{z \in \CC: z - 1 \in W_1^*\}$, and $\tilde K = W_2^* - 1:= \{z \in \CC: z + 1 \in W_2^*\}$, where $W_j^*$ denotes the closure of $W_j$. Since the Brownian Hardy number is translation invariant, we have $bh(W_1) = bh(K)$, $bh(W_2) = bh(\tilde K)$. So long as $p \leq \min\{bh(D), bh(\tilde D)\}$ the first two criteria of Theorem 3.2 are satisfied. Define $V\subset C$ to be the resultant curve after removing infinitely many regularly spaced line segments from $C$. To verify the third criterion, first let $R_1, R_2$ be the upper and lower rays of $V$. For any $\eps>0$, we can choose $a=re^{i\alpha}+1$ sufficiently far from $R_2$ so that 
$P_a(T(U) \leq T(K^c)\;|\;T(R_2) = T(U))\leq \eps$. That is, this contribution to 
$P(T(U)\leq T(K^c))$ can be made arbitrarily small. Taking $r\to\infty$ the quantity $P_a(T(U) \leq T(K^c)\;|\;T(R_1) = T(U))$ tends to a periodic
harmonic function due to the fact $V$ is periodic on each ray. It should
be clear from this, that there is no $r$ value for which $P_a(T(U) \leq T(K^c)\;|\;T(R_1) = T(U))$ vanishes, and hence the third criterion is satisfied. Hence, in this case $bh(T(D)) = \frac{\pi}{4(\alpha-\pi)}$
\\
\begin{figure}[H] \label{}
\begin{center}
 \includesvg[width=10cm, height=8cm]{wedge.svg}
 \caption{}
\end{center}
\end{figure}
{\bf Remark:} As mentioned above, the stopping time technique here is similar to one used in \cite{mecomb}, and purely analytic methods were used in \cite{karafyllia2021hardy, karafyllia2022range} to prove and extend these results. It would be interesting to see whether analytic techniques would apply to prove the results in this section.

\section{Exit time from a lattice} \label{periodic}
We move on to another family of domains where the Brownian Hardy number can be
determined by using an approach similar to the one in the previous section. We
first define a lattice to be the set $\{mz+nz:m,n\in \ZZ\}$ for some pair of
complex numbers $z,w$. So a lattice is a doubly periodic, discrete subset of
$\CC$. Define the domain 
\begin{align*} D =
\CC\setminus\left(\;\bigcup_{z_1\in\Lambda} \{|z-z_1|<r\} \right), 
\end{align*}
where $r$ is chosen such $\{|z-z_1|<r\}\cap \{|z-z_2|<r\} = \emptyset$ for
all $z_1, z_2\in\Lambda$. This ensures that $D$ has one component.
We should expect $bh(T(D)) = \infty$ as Brownian motion is unable
to wander off away from $\partial D$. We will first prove this, and then see how the proof generalizes 
to more general domains of this type.
\begin{theorem}\label{lattice}
    Let $\Lambda$ be a lattice in $\C$, and let $D$ be as defined above. Then $bh(T(D)) = \infty$
\end{theorem}
\begin{proof}
    Let $\{z_1, w_1\}$ denote the segment between 
$z_1$ and $w_1$.
Define $\mc{L} := \{\{\tilde z, \tilde w\}: \tilde z, \tilde w \in \Lambda \text{ and } \tilde z - \tilde w = w \text{ or }\tilde z - \tilde w = z\}$,
meaning the set of edges which form a grid in $\CC$.
Denote the grid by $L = \bigcup_{e \in \mc{L}}e$. Define the process $l_t$ taking values in $\mc{L}\cup \emptyset$ by
$$
    l_t(\omega) = \{e\in\mc{L}: \text{where $e$ is the last edge hit in the history of $\omega(0,t)$}\},
$$
and $l_t(\omega) = \emptyset$ if $\omega$ has not hit an edge yet.
Let $\tau_j = \inf\{t>\tau_{j-1}: Z_t \in L \setminus l_t\}$, with $\tau_0 = 0$. Let $\tilde \tau_j = \tau_j \wedge T(D)$. Then, by Minkowski's inequality,
$$
    \Vert \tilde \tau_n \Vert_p^p = \left\Vert \sum_{j=0}^n (\tilde\tau_j - \tilde\tau_{j-1})\right\Vert_p^p\leq
      \sum_{j=0}^n \left\Vert\tilde\tau_j - \tilde\tau_{j-1}\right\Vert_p^p.
$$
The random variable $\til\tau_j - \til \tau_{j-1}$ is equal to the exit time of Brownian motion from a bounded region. Therefore
$E[|\til\tau_j - \til \tau_{j-1}|^p]<\infty$ for all $p$. 
$$
    \Vert\til\tau_j - \til \tau_{j-1}\Vert_p^p \leq E_a[(\tau_j - \tau_{j-1})^p|\tau_{j-1} < T(D)] P_a(\tau_{j-1} < T(D)).
$$
Here we choose $a\in L\cap D$ which maximizes the quantity on the right. The maximum exists as the region is bounded. The event $\{\tau_{j-1} < T(D)\}$ encodes the probability 
Brownian motion hits $j-1$ edges
without hitting a disk. Let $\alpha := \sup_{z\in L\cap D}P_a(\tau_2 < T(D))$, 
this quantity is strictly less than $1$ because the disks inhibit Brownian motion
started on $L$ from immediately hitting another edge without exiting the domain first.
$$
    \Vert\til\tau_j - \til \tau_{j-1}\Vert_p^p \leq C_p \alpha^j
$$
Therefore,
$$
   \Vert T(D)\Vert_p^p = \lim_{j\to\infty} \Vert \til \tau_n \Vert_p^p \leq \sum_{j=1}^\infty C_p \alpha^j  < \infty.
$$
Since $C_p<\infty$ for all $p$, we have $bh(T(D)) = \infty$.
\end{proof}
Again, if we extract only the parts of the proof that were needed we can generalize the result.
We required that the sets we remove at each vertex not become arbitrarily small. 
This ensures $\sup_{a\in L\cap D}P_a(\tau_2 < T(D)) < 1$, which was required so that the final bound was finite. 
For the stopping times in the proof, it was necessary to consider the first exit time from the 
grid with the most recently visited segment removed. 
This exit time is equal to the first exit from the polygon remaining after removing that segment. 
The proof relied on the boundedness of this polygon to ensure the finiteness of all moments. 

We must also note the role of our particular choice of graph induced from the lattice. The main importance was the maximum possible sized component 
of $\CC\setminus (L\setminus e)$ where $e\in\mc{L}$. We could remove countably many edges, and the result would remain the same, provided the maximum sized component was bounded. However, if we removed
all the edges to the left of the origin, the proof would instead yield $bh(D) \geq 1/2$. This is due to the property that the largest component of $\CC\setminus L$ is now the left half-plane.
\begin{theorem}
Let $\{U\}_{j\in \NN}$ be a collection of domains which contain the origin. Let $C =\{c_j\}_{j\in \NN}$ be a discrete subset of $\CC$ such that for all $i, j$ 
with $i\neq j$ it holds that $(U_i + c_i)\cap (U_j + c_j) = \emptyset$. Consider the domain $D$ given by,
$$
    D = \CC \setminus \bigcup_{j\in\NN} (U_j + c_j).
$$
Let $\mc{L}$ be a graph embedded in $\CC$ whose vertices are a subset of $C$, let $L$ be the union of all the edges in $\mc{L}$.
If for $p\in\RR^+$ there exists some $\delta < 1$ and $N$ such that for every edge $e\in\mc{L}$, the following conditions hold.
\begin{itemize}
    \item Define $\tau = \inf_{t>0}\{t: B_t \in L \setminus e\}$, it holds that $\sup_{a\in e\cap D} P_a(\tau < T(D)) \leq \delta$.
    \\
\item Let $U$ be the component of $\CC\setminus (L\setminus e)$   
containing $e$, it holds that $\sup_{a\in e\cap D} E_a[\tau_U^p] \leq N$.
\end{itemize}
Then $bh(D) \geq p$.
\end{theorem}

\begin{proof}
    The proof follows identically to the proof of Theorem \ref{lattice}, with the lattice and grid defined as above.
\end{proof}
The condition above can be written more concisely 
\begin{align*}
    \sup_{e\in\mc{L}}\sup_{a\in e\cap D} P_a(\tau<T(D)) <&\; 1,\\
    \sup_{e\in\mc{L}}\sup_{a\in e\cap D} E_a[\tau_U^p] <\;& \infty,
\end{align*}
for $\tau$ and $U$ as defined above.\\

For an example, let $W = \{z:-\alpha\leq Arg(z)\leq\alpha\}$, let $C = \Z^2 \cap W^c$ and let $\mc{L}$ be the subgraph of the $\Z^2$ lattice with 
vertices at $C$. For the collection of domains we let $U_j := B_r(0)$ for all $j$, where $r < 1$. Let $D$ be as above,
$$
    D = \CC \setminus \bigcup_{c \in C} B_r(c).
$$
Then the above theorem implies $bh(D) = bh(W) = \frac{\pi}{4\alpha}$.

\section{Construction of wedge-like domains with interesting properties} \label{achievecrit}

For many of the previous examples, and indeed for virtually all known examples with $0 < bh(D) < \infty$, it is true that $E[T(D)^{bh(D)}] = \infty$. This raises the following question: does there exist a domain $D$ with $0 < bh(D) < \infty$ such that the $bh(D)$-th moment of $T(D)$
is finite? The answer is affirmative, there exists such a simply-connected example for any $bh(D) \geq 1/4$. Any simply connected domain satisfies $bh(D) \geq 1/4$ (\cite{burk}), so there is no such simply connected domain for smaller exponents, however if $0<p\leq \frac{1}{4}$ one can still construct a stopping time $\tau\in\mc{H}$ which works, although it is no longer the exit time of a domain. In particular, we have the following.

\begin{theorem}\label{example_finite}
For all $p\geq\frac{1}{4}$ there exists a domain $D$ such that $bh(D) = p$ and $E[T(D)^p] < \infty$. 
If $0<p\leq \frac{1}{4}$ there exists a stopping time $\tau\in\mc{H}$ such that the same is true.
\end{theorem}
\begin{proof}
We will prove the result for $p=1$. For general $p$ the result will follow from this, although as indicated above for $p \leq \frac{1}{4}$ it is no longer the exit time of a domain (more explanation will be given below). Let $D:=\{z:-\frac{\pi}{4} \leq Arg z < \frac{\pi}{4}\}$ be the right quarter-plane. We note
$bh(D) = 1$ and $E[T(D)] = \infty$. Let $\{\alpha_n\}_{n\geq0}$ be a sequence of strictly positive real numbers which monotonically increase to $\pi/4$. Let $W_n = \{z : -\alpha_n \leq Arg z < \alpha_n\}$. $\{R_n\}_{n\geq0}$ will be a sequence of positive real numbers monotonically increasing to $\infty$, whose values will be defined shortly. Define $D_1 = W_1$ and (see Figure \ref{growing wedges} below)
$$
    D_n = \left(\{z: |z|\geq R_n\} \cap W_n\right) \cup D_{n-1}.
$$
Noting that $bh(W_n) = \pi/(4\alpha_n)<1$, we have that $E_1[\tau_{D_1}] = C < \infty$. From here we will fix Brownian motion to start at $1$ as this point is common to all domains. 

Now choose $R_2$ such that $E_1[T(D_2)] \leq C(1+\frac{1}{2})$; we can find such an $R_2$, because 
if we take $R_2\to\infty$ it follows that $T(D_2)\to T(D_1)$, and as $T(D_2)\leq T(W_2)\in L^1$, by the dominated convergence theorem it holds that $E_1[T(D_2)] \to E[T(D_1)]$. 

This indicates the manner that the sequence $R_n$ is chosen. In general, we choose $R_n$ such that,
$$
    E[T(D_n)] \leq C(1+\frac{1}{2} + \ldots + \frac{1}{2^n} )
$$
This is possible, because by the same argument as before we have $E[T(D_n)] \to E[T(D_{n-1})]$ when $R_n$ is taken to $\infty$. The $D_n$'s are increasing in size, so we may let $D_\infty = \lim_{n\to\infty} D_{n}$. It follows by the monotone convergence theorem that $E[T(D_\infty)] = \lim_{n \to \infty} E[T(D_n] \leq 2C < \infty$. Furthermore, $D_\infty$ contains a translation of every wedge $W_n$, so it follows that $bh(W_n) \geq bh(D_\infty)$ for every $n$. This proves that $1\geq bh(D_\infty)$, but we also know $E[T(D_\infty)]<\infty$, so indeed
$bh(D_\infty) = 1$. This completes the proof for $p=1$.

\if2
The case of $p < 1/4$ corresponds to wedges $W_\alpha$ with $\alpha > \pi$. In this case, we proceed as above. This is equivalent to determining the 
Brownian hardy number for $\tau\in\mc{H}$ being the projection of the stopping time of $\tau_{\DD}$ through $(1+z^{2/\pi \alpha})/(1-z^{2/\pi\alpha})$
\fi

We may proceed in exactly the same manner for general $p$, using a wedge of angle $\frac{\pi}{2p}$ in place of the one of $\frac{\pi}{2}$. Alternatively, we may project the exit time of $D_\infty$ by the analytic function $z \to z^{1/p}$ in order to achieve the same effect. If $p>1/4$ this will still project to the exit time of a domain, but if $0 < p \leq 1/4$ the domain will essentially wrap over itself, and the resulting stopping time will no longer be the exit time of a domain. 
\end{proof}


\begin{figure}[H] 
\begin{center}
\includesvg[width=\textwidth, height=8cm]{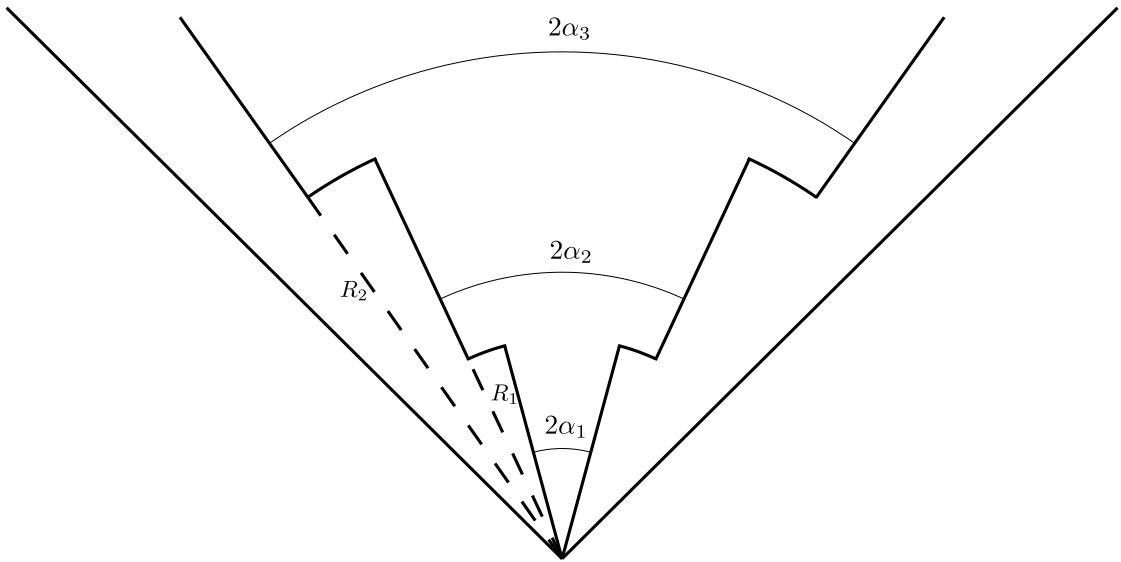}
\caption{An example of $D_n$, rotated for formatting.}\label{growing wedges}
\end{center}
\end{figure}

The domain $D_\infty$ constructed above is interesting for another reason, and that is that it gives the example of a domain whose exit time has finite expectation, but where the expectation does not have quadratic growth as a function of the initial point (see the discussion at the end of the introduction, Section \ref{intro}). In particular, we have the following result.

\if2 Suppose $V$
is a domain with $u(x) := E_x[T(V)] <\infty$. It follows by Dynkin's theorem that $u$ satisfies
the PDE $\Delta u = -2$ with $u\equiv0$ on $\partial V$. If $V$ is unbounded, this PDE does not
have a unique solution. However, if we suppose one more condition, then the solution must be the
expected exit time from $x$. The following theorem is due to Burkholder \cite{burk}.
\begin{theorem}
    Let $R$ be a domain with $E[T(R)]< \infty$. Suppose $u:R\to\RR$ satisfies $\Delta u = -2$
    in $R$ and $u\equiv 0$ on $\partial R$, and there exists some constant $C>1$ with
    $$
        |u(z)| \leq C(1+|z|^2).
    $$
    Then $u(z) = E_z[T(R)]$.
\end{theorem}
The domain $D_\infty$ shows that the condition is sufficient, but not necessary. 
\fi

\begin{proposition}
    Let $u(a) = E_a[T(D_\infty)]$. Then there does not exist a constant $C\geq1$ such that $|u(a)| \leq C(1+|a|^2)$.
\end{proposition}
\begin{proof}
Let $\alpha_n, R_n, D_n$ be the sequences defined in the proof of Theorem 
\ref{example_finite}, and recall that
$W_n:=\{z:-\alpha_n\leq Argz \leq \alpha_n\}$. The boundary of $W_n$ is a subset of the lines defined by $y=\pm m_n x$, where $m_n = \tan \alpha_n \in (0,1)$, and it is straightforward to check that $h(x,y) = \frac{m_n^2 x^2 - y^2}{1-m_n^2}$ satisfies the conditions of Theorem \ref{burk_thm}, and thus 

$$
 E_{x+yi}[T(W_n)] = \fr{x^2 m^2_n - y^2}{1-m^2_n} = \fr{x^2\tan^2\alpha_n - y^2}{1-\tan^2\alpha_n}.
$$
If $a = x$ on the real line, then $ E_{x}[T(W_n)] =x^2\tan^2\alpha_n/(1-\tan^2\alpha_n) $. 
Let $\tilde W_n  := W_n \cap \{|z| = R_n\}$. Since $T(\tilde W_n) \leq T(W_n)$, by the strong Markov property, we have

\begin{equation} \label{vicec}
    E_{2R_n}[T(W_n)] \leq E_{2R_n}[T(\tilde W_n)] + E_{R_n}[T(W_n)].    
\end{equation}

This can be seen path-wise: if the path exits $\tilde W_n$ on $\partial W_n$ then $T(\tilde W_n) = T(W_n)$; otherwise, if the path exits $\tilde W_n$ at $R_n e^{i \theta} \in \{|z| > R_n\}$, then we have 

$$
E_{R_ne^{i \theta}}[T(W_n)] = \fr{R_n^2\cos^2\theta \tan^2\alpha_n - R_n^2 \sin^2 \theta}{1-\tan^2\alpha_n} \leq \fr{R_n^2\tan^2\alpha_n}{1-\tan^2\alpha_n} = E_{R_n}[T(W_n)].
$$ 

In other words, the exit time from
$W_n$ can be bounded above by restarting the path at $R_n$ upon exiting $\tilde W_n$. In addition, it holds that $E_{R_n}[T(W_n)] = (1/4) E_{2R_n}[T(W_n)]$. This is a special case of the identity $E_{a}[T(W_n)] = \frac{1}{r^2}E_{ra}[T(W_n)]$ for $r>0$, which follows from the fact that $W_n$ is invariant under the dilation $z \to rz$ and the time scaling of Brownian motion. Thus, from \eqref{vicec}, we obtain
$$
    \fr{3}{4}E_{2R_n}[T(W_n)] \leq E_{2R_n}[T(\tilde W_n)] \leq E_{2R_n}[T(D_n)]\leq E_{2R_n}[T(D_\infty)].
$$
The last set of inequalities holds from the inclusions $\tilde W_n \subset D_n \subset D_{\infty}$. Finally
$$
    \fr{3}{4} \fr{E_{2R_n}[T(W_n)]}{4R_n^2} = \fr{3\tan^2\alpha_n}{16(1-\tan^2\alpha_n)} \leq \fr{E_{2R_n}[T(D)]}{4R_n^2}.
$$
Letting $\alpha_n\nearrow\pi/4$, we find $\lim_{n\to\infty}\fr{E_{2R_n}[T(D)]}{4R_n^2} = \infty$. This completes the proof.
\end{proof}

\section{Acknowledgements}

The authors would like to thank Xi Geng, Lesley Ward, Kaustav Das, and Binghao Wu for useful conversations.

\bibliographystyle{plainurl}
\bibliography{citation}

\begin{thebibliography}{10}

\bibitem{banuelos1993conditioned}
R.~Ba{\~n}uelos and T.~Carroll.
\newblock Conditioned {B}rownian motion and hyperbolic geodesics in simply connected domains.
\newblock {\em The Michigan Mathematical Journal}, 40(2):321--332, 1993.

\bibitem{banuelosdrum}
R.~Ba{\~n}uelos and T.~Carroll.
\newblock Brownian motion and the fundamental frequency of a drum.
\newblock {\em Duke Mathematical Journal}, 75(3):575--602, 1994.

\bibitem{banuelos2005sharp}
R.~Ba{\~n}uelos and T.~Carroll.
\newblock Sharp integrability for {B}rownian motion in parabola-shaped regions.
\newblock {\em Journal of Functional Analysis}, 218(1):219--253, 2005.

\bibitem{banuelos2001first}
R.~Ba{\~n}uelos, R.~D. DeBlassie, and R.~Smits.
\newblock The first exit time of planar {B}rownian motion from the interior of a parabola.
\newblock {\em The Annals of Probability}, 29(2):882--901, 2001.

\bibitem{banuelos2024conjecture}
R.~Ba{\~n}uelos and P.~Mariano.
\newblock On a conjecture of a {P}\'olya functional for triangles and rectangles.
\newblock {\em arXiv preprint arXiv:2406.01778}, 2024.

\bibitem{banuelos2020bounds}
R.~Ba{\~n}uelos, P.~Mariano, and J.~Wang.
\newblock Bounds for exit times of {B}rownian motion and the first {D}irichlet eigenvalue for the {L}aplacian.
\newblock {\em Transactions of the American Mathematical Society}, 376(08):5409--5432, 2023.

\bibitem{banuelos1997brownian}
R.~Ba{\~n}uelos and R.~Smits.
\newblock Brownian motion in cones.
\newblock {\em Probability theory and related fields}, 108:299--319, 1997.

\bibitem{banuelos2002torsional}
R.~Ba{\~n}uelos, M.~Van~den Berg, and T.~Carroll.
\newblock Torsional rigidity and expected lifetime of {B}rownian motion.
\newblock {\em Journal of the London Mathematical Society}, 66(2):499--512, 2002.

\bibitem{barton2014new}
A.~Barton and L.~Ward.
\newblock A new class of harmonic measure distribution functions.
\newblock {\em The Journal of Geometric Analysis}, 24:2035--2071, 2014.

\bibitem{bass}
R.~Bass.
\newblock {\em Probabilistic techniques in analysis}.
\newblock Springer Science \& Business Media, 1995.

\bibitem{dimitri}
D.~Betsakos, M.~Boudabra, and G.~Markowsky.
\newblock On the probability of fast exits and long stays of planar {B}rownian motion in simply connected domains.
\newblock {\em Journal of Mathematical Analysis and Applications}, 2020.

\bibitem{betsakos2022duration}
D.~Betsakos, M.~Boudabra, and G.~Markowsky.
\newblock On the duration of stays of {B}rownian motion in domains in {E}uclidean space.
\newblock {\em Electronic Communications in Probability}, 27:1--12, 2022.

\bibitem{boudbuttme}
M.~Boudabra, A.~Buttigieg, and G.~Markowsky.
\newblock A collection of results relating the geometry of plane domains and the exit time of planar {B}rownian motion.
\newblock {\em Computational Methods and Function Theory}, 23(3):469--488, 2023.

\bibitem{usexitmax}
M.~Boudabra and G.~Markowsky.
\newblock Maximizing the $p$-th moment of exit time of planar {B}rownian motion from a given domain.
\newblock {\em Journal of Applied Probability}, 57(4):1135--1149, 2020.

\bibitem{mecomb}
M.~Boudabra and G.~Markowsky.
\newblock The $p$-th moment of exit time of planar {B}rownian motion on comb domains.
\newblock {\em Annales Fennici Mathematici}, 46(1), 2021.

\bibitem{burk}
D.L. Burkholder.
\newblock Exit times of {B}rownian motion, harmonic majorization, and {H}ardy spaces.
\newblock {\em Advances in Mathematics}, 26(2):182--205, 1977.

\bibitem{carroll2001old}
T.~Carroll.
\newblock Old and new on the bass note, the torsion function and the hyperbolic metric.
\newblock {\em Irish Math. Soc. Bull}, 47:41--65, 2001.

\bibitem{carroll2006brownian}
T.~Carroll.
\newblock Brownian motion and harmonic measure in conic sections.
\newblock In {\em Potential Theory in Matsue}, pages 25--41. Mathematical Society of Japan, 2006.

\bibitem{coffee}
M.W. Coffey.
\newblock Expected exit times of {B}rownian motion from planar domains: Complements to a paper of {M}arkowsky.
\newblock {\em {\it arXiv:1203.5142}}, 2012.

\bibitem{crowdy2020solving}
D.~Crowdy.
\newblock {\em Solving problems in multiply connected domains}.
\newblock SIAM, 2020.

\bibitem{crowdy2016schottky}
D.~Crowdy, E.~Kropf, C.~Green, and M.~Nasser.
\newblock The {S}chottky--{K}lein prime function: a theoretical and computational tool for applications.
\newblock {\em IMA Journal of Applied Mathematics}, 81(3):589--628, 2016.

\bibitem{davis}
B.~Davis.
\newblock Brownian motion and analytic functions.
\newblock {\em The Annals of Probability}, 7(6):913--932, 1979.

\bibitem{deblassie1987exit}
R.~DeBlassie.
\newblock Exit times from cones in $\mathbb{R}^n$ of {B}rownian motion.
\newblock {\em Probability theory and related fields}, 74(1):1--29, 1987.

\bibitem{de1988remark}
R.~DeBlassie.
\newblock Remark on exit times from cones of {B}rownian motion: {P}rob. {T}h. {R}el. {F}ields 74, 1--29 (1987).
\newblock {\em Probability theory and related fields}, 79(1):95--97, 1988.

\bibitem{durBM}
R.~Durrett.
\newblock {\em Brownian motion and martingales in analysis}.
\newblock Wadsworth Advanced Books \& Software, 1984.

\bibitem{green2024towards}
C.~Green and M.~Nasser.
\newblock Towards computing the harmonic-measure distribution function for the middle-thirds {C}antor set.
\newblock {\em Journal of Computational and Applied Mathematics}, 448:115903, 2024.

\bibitem{green2022harmonic}
C.~Green, M.~Snipes, L.~Ward, and D.~Crowdy.
\newblock Harmonic-measure distribution functions for a class of multiply connected symmetrical slit domains.
\newblock {\em Proceedings of the Royal Society A}, 478(2259):20210832, 2022.

\bibitem{karafyllia2021property}
C.~Karafyllia.
\newblock On a property of harmonic measure on simply connected domains.
\newblock {\em Canadian Journal of Mathematics}, 73(2):297--317, 2021.

\bibitem{karafyllia2021hardy}
C.~Karafyllia.
\newblock On the {H}ardy number of comb domains.
\newblock {\em Annales Fennici Mathematici}, 47:587--601, 2022.

\bibitem{karafyllia2022range}
C.~Karafyllia.
\newblock The range of {H}ardy numbers for comb domains.
\newblock {\em Computational Methods and Function Theory}, 22(4):743--753, 2022.

\bibitem{kim2021quantitative}
D.~Kim.
\newblock Quantitative inequalities for the expected lifetime of {B}rownian motion.
\newblock {\em Michigan Mathematical Journal}, 70(3):615--634, 2021.

\bibitem{mahenthiram2024harmonic}
A.~Mahenthiram.
\newblock Harmonic-measure distribution functions of simply connected and doubly connected polygonal domains.
\newblock {\em Journal of Mathematical Analysis and Applications}, 537(1):128308, 2024.

\bibitem{mahenthiram2025computing}
A.~Mahenthiram.
\newblock Computing harmonic-measure distribution functions of some multiply connected unbounded planar domains.
\newblock {\em Journal of Mathematical Analysis and Applications}, page 129291, 2025.

\bibitem{meecp}
G.~Markowsky.
\newblock On the expected exit time of planar {B}rownian motion from simply connected domains.
\newblock {\em Electronic Communications in Probability}, 16:652--663, 2011.

\bibitem{markowsky2013method}
G.~Markowsky.
\newblock A method for deriving hypergeometric and related identities from the ${H}^2$ hardy norm of conformal maps.
\newblock {\em Integral Transforms and Special Functions}, 24(4):302--313, 2013.

\bibitem{mejmaa}
G.~Markowsky.
\newblock The exit time of planar {B}rownian motion and the {P}hragm{\'e}n--{L}indel{\"o}f principle.
\newblock {\em Journal of Mathematical Analysis and Applications}, 422(1):638--645, 2015.

\bibitem{markowsky2017planar}
G.~Markowsky.
\newblock On the planar {B}rownian {G}reen's function for stopping times.
\newblock {\em Journal of Mathematical Analysis and Applications}, 455(2):1221--1233, 2017.

\bibitem{Markowsky_2018}
G.~Markowsky.
\newblock On the distribution of planar {B}rownian motion at stopping times.
\newblock {\em Annales Fennici Mathematici}, 43(2):597–616, 2018.

\bibitem{markowsky2018remark}
G.~Markowsky.
\newblock A remark on the probabilistic solution of the {D}irichlet problem for simply connected domains in the plane.
\newblock {\em Journal of Mathematical Analysis and Applications}, 464(2):1143--1146, 2018.

\bibitem{mac}
T.R. McConnell.
\newblock The size of an analytic function as measured by {L}\'evy's time change.
\newblock {\em The Annals of Probability}, 13(3):1003--1005, 1985.

\bibitem{snipes2005realizing}
M.~Snipes and L.~Ward.
\newblock Realizing step functions as harmonic measure distributions of planar domains.
\newblock {\em Annales Fennici Mathematici}, 30(2):353--360, 2005.

\bibitem{snipes2008convergence}
M.~Snipes and L.~Ward.
\newblock Convergence properties of harmonic measure distributions for planar domains.
\newblock {\em Complex Variables and Elliptic Equations}, 53(10):897--913, 2008.

\bibitem{snipes2016harmonic}
M.~Snipes and L.~Ward.
\newblock Harmonic measure distributions of planar domains: a survey.
\newblock {\em The Journal of Analysis}, 24:293--330, 2016.

\end{thebibliography}

\end{document}